\documentclass[12pt,reqno]{article}

\usepackage[usenames]{color}
\usepackage[colorlinks=true,
linkcolor=webgreen, filecolor=webbrown,
citecolor=webgreen]{hyperref}

\definecolor{webgreen}{rgb}{0,.5,0}
\definecolor{webbrown}{rgb}{.6,0,0}

\usepackage{amssymb}
\usepackage{graphicx}
\usepackage{amscd}
\usepackage{lscape}
\usepackage{tikz}
\usepackage{tikz-cd}
\usepackage{pgfplots}

\usetikzlibrary{matrix}
\usetikzlibrary{fit,shapes}
\usetikzlibrary{positioning, calc}
\tikzset{circle node/.style = {circle,inner sep=1pt,draw, fill=white},
        X node/.style = {fill=white, inner sep=1pt},
        dot node/.style = {circle, draw, inner sep=5pt}
        }
\usepackage{tkz-fct}

\usepackage{amsthm}
\newtheorem{theorem}{Theorem}

\newtheorem{proposition}[theorem]{Proposition}

\theoremstyle{definition}

\newtheorem{example}[theorem]{Example}

\usepackage{float}

\usepackage{graphics,amsmath}
\usepackage{amsfonts}
\usepackage{latexsym}
\usepackage{epsf}

\DeclareMathOperator{\Rev}{Rev}

\newcommand{\seqnum}[1]{\href{http://oeis.org/#1}{\underline{#1}}}

\begin{document}

\begin{center}
\vskip 1cm{\LARGE\bf Series reversion with Jacobi and Thron continued fractions} \vskip 1cm \large
Paul Barry\\
School of Science\\
Waterford Institute of Technology\\
Ireland\\
\href{mailto:pbarry@wit.ie}{\tt pbarry@wit.ie}
\end{center}
\vskip .2 in

\begin{abstract} Using ordinary and exponential generating functions, we explore the reversion of power series defined by $2$nd order recurrences. We express the reversions in terms of Jacobi and Thron continued fractions. We find relations with Eulerian expressions using a transformation of continued fractions.
\end{abstract}

\section{Introduction} Since the work of Flajolet \cite{Flajolet} there has been much interest in the interplay between the theory of Jacobi continued fractions \cite{Wall} and certain topics in combinatorics. Much less attention has been paid to Thron-type continued fractions. Some recent publications \cite{Oste, Sokal, Willerton} have begun to address this issue. In this note, we look at interesting connections between series inversion, ordinary and exponential generating functions, and Jacobi and Thron-type continued fractions. 

A simple example of what follows starts by considering the generating function $g(x)=1-x$ of the sequence $(1,-1,0,0,0,\ldots)$. The revert transform of this generating function is given by $c(x)=\frac{1-\sqrt{1-4x}}{2x}$, the generating function of the Catalan numbers. This generating function has the Jacobi continued fraction expression
$$\cfrac{1}{1-x-\cfrac{x^2}{1-2x-\cfrac{x^2}{1-2x-\cfrac{x^2}{1-2x-\cdots}}}},$$ as well as the (trivial) Thron-type continued fraction
$$\cfrac{1}{1-0.x-\cfrac{x}{1-0.x-\cfrac{x}{1-0.x-\cdots}}},$$ more usually rendered as the Stieltjes continued fraction
$$\cfrac{1}{1-\cfrac{x}{1-\cfrac{x}{1-\cdots}}}.$$
On the other hand, $1-x=1-\frac{x}{1!}$ in this case also coincides with the exponential generating function of the sequence $(1,-1,0,0,0,\ldots)$. We write this as $g_e(x)=1-x$. Now $\frac{1}{g_e(x)}=\frac{1}{1-x}$ as an exponential generating function expands to give the factorial numbers $n!$, providing a link between factorial numbers $n!$ and the Catalan numbers $C_n$. This link can be explored in terms of continued fractions by noting that the ordinary generating function of $n!$ can be expressed as the Jacobi continued fraction \cite{Flajolet}
$$\cfrac{1}{1-x-\cfrac{x^2}{1-3x-\cfrac{4x^2}{1-5x-\cfrac{9x^2}{1-7x-\cdots}}}}. $$
We now observe that by taking the first differences of the coefficients of $x$, and by dividing successive coefficients of $x^2$ by $1,4,9,16,\ldots$ respectively, we recover the Jacobi-type continued fraction for $c(x)$. We wish to explore this phenomenon in a more general setting.
The link between $1-x$ and $n!$ can be seen in the context of exponential Riordan arrays by noting that
$$[\frac{1}{1-x}, x]^{-1}=[1-x, x].$$ In matrix terms, we have (using a suitable truncation) that
$$\left(
\begin{array}{cccccc}
 1 & 0 & 0 & 0 & 0 & 0 \\
 1 & 1 & 0 & 0 & 0 & 0 \\
 2 & 2 & 1 & 0 & 0 & 0 \\
 6 & 6 & 3 & 1 & 0 & 0 \\
 24 & 24 & 12 & 4 & 1 & 0 \\
 120 & 120 & 60 & 20 & 5 & 1 \\
\end{array}
\right)^{-1} = \left(
\begin{array}{cccccc}
 1 & 0 & 0 & 0 & 0 & 0 \\
 -1 & 1 & 0 & 0 & 0 & 0 \\
 0 & -2 & 1 & 0 & 0 & 0 \\
 0 & 0 & -3 & 1 & 0 & 0 \\
 0 & 0 & 0 & -4 & 1 & 0 \\
 0 & 0 & 0 & 0 & -5 & 1 \\
\end{array}
\right).$$
The association sought is seen in the first columns of these inverse matrices.

We note that many of the sequences we shall encounter are recorded in the On-Line Encyclopedia of Integer Sequences (OEIS) \cite{SL1, SL2} and where this is the case, we record such sequences under their OEIS number.

\section{Preliminaries}
We let $\mathcal{F}=\{ a_0+a_1 x+ a_2 x^2+ \cdots \,| a_i \in \mathbf{R}\}$ be the set of formal power series with coefficients $a_i$ drawn from the ring $\mathbf{R}$. This ring can be any ring over which the operations we will carry out make sense, but for concreteness it can be assumed to be $\mathbb{Q}$. For combinatorial problems, the ring $\mathbf{R}\}$ is often the ring of integers $\mathbb{Z}$.  We shall use two distinguished subsets of $\mathcal{F}$, namely
$$\mathcal{F}_0=\{ a_0+a_1 x+ a_2 x^2+ \cdots \,| a_i \in \mathbf{R}, a_0 \ne 0\},$$ and
$$\mathcal{F}_1=\{ a_1 x+ a_2 x^2+ \cdots \,| a_i \in \mathbf{R}, a_1 \ne 0\}.$$ Elements of $\mathcal{F}_1$ are composable and possess compositional inverses. If $f(x) \in \mathcal{F}_1$, we shall denote by $\bar{f}(x)$ or $\Rev(f)(x)$ its compositional inverse. Thus we have $\bar{f}(f(x))=x$ and $f(\bar{f}(x))=x$. For $g(x) \in
\mathcal{F}_0$, we shall say that the power series $\frac{1}{x}\Rev(xg(x))$ is the reversion of $g(x)$, or that it is the \emph{revert transform} of $g(x)$.

If $g(x) \in \mathcal{F}_0$ has a Jacobi continued fraction expression of the form
$$g(x)=\cfrac{a_0}{1-\alpha_1 x - \cfrac{\beta_1 x^2}{1-\alpha_2 x - \cfrac{\beta_2 x^2}{1-\alpha_3 x -\cfrac{\beta_3 x^2}{1-\alpha_4 x - \cdots}}}},$$ we shall designate as its $\mathbb{T}$ \emph{transform}, written as $\mathbb{T}(g)(x)$, the power series with Jacobi continued fraction expression
$$\mathbb{T}(g)(x)=\cfrac{a_0}{1-\alpha_1 x - \cfrac{\frac{\beta_1}{1} x^2}{1-(\alpha_2-\alpha_1) x - \cfrac{\frac{\beta_2}{4} x^2}{1-(\alpha_3-\alpha_2) x - \cfrac{\frac{\beta_3}{9}x^2}{1-(\alpha_4-\alpha_3)-\cdots}}}}.$$
Note that we sometimes use the shorthand $\mathcal{J}(\alpha_1, \alpha_2, \ldots; \beta_1, \beta_2,\ldots)$ for the Jacobi continued fraction representing $g(x)$ as above, where we assume $a_0=1$.

\section{Main result}
We consider the three-parameter generating function
$$g(x)=\frac{1+ax}{1+bx +c x^2}.$$
We then have the following result.
\begin{proposition}\label{One} We have
$$g(x)=\Rev\left(\mathbb{T}\left(\frac{1}{g_e(x)}\right)\right),$$ where $g_e(x)$ is the exponential generating function of the sequence with ordinary generating function $g(x)$.
\end{proposition}
Implicit in this proposition is the statement that the ordinary generating function corresponding to the  exponential generating function can be expressed as a Jacobi continued fraction.
We state this explicitly.
\begin{proposition} The sequence with exponential generating function $\frac{1}{g_e(x)}$ has an ordinary generating function given by
$$\cfrac{1}{1-(b-a)x-\cfrac{(a^2-ab+c)x^2}{1-(2b-3a)x-\cfrac{4(a^2-ab+c)x^2}{1-(3b-5a)x-\cfrac{9(a^2-ab+c)x^2}{1-\cdots}}}}.$$
\end{proposition}
\begin{proof}
We have
$$g_e(x)=\mathcal{L}_t^{-1}\left(\frac{1}{t}g\left(\frac{1}{t}\right)\right)(x),$$ where $\mathcal{L}$ is the Laplace transform.
We find that
$$g_e(x)=\frac{e^{-\frac{1}{2} t \left(\sqrt{b^2-4 c}+b\right)} \left(2 a \left(e^{t \sqrt{b^2-4 c}}-1\right)+b \left(-e^{t \sqrt{b^2-4 c}}\right)+\sqrt{b^2-4 c} \left(e^{t \sqrt{b^2-4
   c}}+1\right)+b\right)}{2 \sqrt{b^2-4 c}},$$ and hence that
\begin{align*}\frac{1}{g_e(x)}&=\frac{2 \sqrt{b^2-4 c} e^{\frac{1}{2} t \left(\sqrt{b^2-4 c}+b\right)}}{(b-2 a) \left(1-e^{t \sqrt{b^2-4 c}}\right)+\sqrt{b^2-4 c} \left(e^{t \sqrt{b^2-4 c}}+1\right)}\\
&=\frac{2 \sqrt{b^2-4 c} e^{\frac{1}{2} t \left(\sqrt{b^2-4 c}+b\right)}}
{(2a-b+\sqrt{b^2-4c})(e^{t \sqrt{b^2-4c}}-1)+2 \sqrt{b^2-4c}}\\
&=\frac{e^{\frac{1}{2} t \left(\sqrt{b^2-4 c}+b\right)}}{1+\frac{1}{2}\left(1+\frac{2a-b}{\sqrt{b^2-4c}}\right)\left(e^{\sqrt{b^2-4c}t}-1\right)}\\
&=\frac{e^{\alpha t}}{1+\beta e^{\gamma t}},\end{align*}
where
\begin{align*} \alpha&=b+\sqrt{b^2-4c}\\
\beta &=\frac{1}{2}\left(1+\frac{2a-b}{\sqrt{b^2-4c}}\right)=\frac{2a-b+\sqrt{b^2-4c}}{2\sqrt{b^2-4c}}\\
\gamma&=\sqrt{b^2-4c}.\end{align*}
We can simplify the above expression further to give
$$\frac{1}{g_e(x)}=e^{\frac{1}{2}(b+\sqrt{b^2-4c})t}.\frac{1-v}{1-v e^{(b-2a)\frac{1-v}{1+v}t}},$$
where $$v=\frac{2a-b+\sqrt{b^2-4c}}{2a-b-\sqrt{b^2-4c}}.$$
This is now in a form to apply results of Flajolet \cite{Flajolet} to arrive at our conclusion. 
\end{proof}
See Figure \ref{Motzkin} for the contributions of the $9$ weighted Motzkin paths of length $4$ to the expansion of $\Rev\left(\frac{1+ax}{1+bx+cx^2}\right)$. We have 
$$[x^4]\Rev\left(\frac{1+ax}{1+bx+cx^2}\right)=
24a^4 - 60a^3 b + 2a^2(25b^2 + 14c) - ab(15b^2 + 38c) + b^4 + 11b^2c + 5c^2.$$

\textbf{Proof of proposition}\,\,\ref{One}.
\begin{proof}By the above proposition, we have that  $\mathbb{T}\left(\frac{1}{g_e(x)}\right)$ can be expressed as the Jacobi continued fraction
$$\cfrac{1}{1-(b-a)x-\cfrac{(a^2-ab+c)x^2}{1-(b-2a)x-\cfrac{(a^2-ab+c)x^2}{1-(b-2a)x-\cfrac{(a^2-ab+c)x^2}{1-\cdots}}}}
$$
This is equal to
$$\frac{1}{1-(b-a)x-(a^2-ab+c)x^2 h(x)},$$ where $h(x)$ satisfies
$$h(x)=\frac{1}{1-(b-2a)x-(a^2-ab+c)x^2 h(x)}.$$
We find that
$$h(x)=\frac{1-(b-2a)x-\sqrt{1+2(2a-b)x+(b^2-4c)x^2}}{2x^2(a^2-ab+c)}.$$  Substituting $h(x)$ back into $\frac{1}{1-(b-a)x-(a^2-ab+c)x^2 h(x)}$, we find that
$$\mathbb{T}\left(\frac{1}{g_e(x)}\right)=\frac{2}{1-bx+\sqrt{1+2(2a-b)x+(b^2-4c)x^2}}.$$
But this last expression is precisely $\Rev(g(x))$.
\end{proof}

\begin{figure}
\begin{center}
\begin{tikzpicture}
\draw (0,0)--(1,1)--(2,2)--(3,1)--(4,0);
\foreach \Point in {(0,0), (1,1),(2,2),(3,1),(4,0)}{
 \fill \Point circle[radius=1pt]; }
   \node[align=center] at (8,0) {$4(a^2-ab+c)^2$};
\draw (0,3)--(1,4)--(2,3)--(3,4)--(4,3);
\foreach \Point in {(0,3), (1,4),(2,3),(3,4),(4,3)}{
 \fill \Point circle[radius=1pt]; }
    \node[align=center] at (8,3) {$(a^2-ab+c)^2$};
\draw (0,5)--(1,5)--(2,5)--(3,6)--(4,5);
\foreach \Point in {(0,5), (1,5),(2,5),(3,6),(4,5)}{
 \fill \Point circle[radius=1pt]; }
   \node[align=center] at (8,5) {$(b-a)^2(a^2-ab+c)$};
\draw (0,7)--(1,7)--(2,8)--(3,8)--(4,7);
\foreach \Point in {(0,7), (1,7),(2,8),(3,8),(4,7)}{
 \fill \Point circle[radius=1pt]; }
  \node[align=center] at (8,7) {$(b-a)(a^2-ab+c)(2b-3a)$};
\draw (0,9)--(1,9)--(2,10)--(3,9)--(4,9);
\foreach \Point in {(0,9), (1,9),(2,10),(3,9),(4,9)}{
 \fill \Point circle[radius=1pt]; }
  \node[align=center] at (8,9) {$(b-a)^2(a^2-ab+c)$};
\draw (0,11)--(1,12)--(2,12)--(3,12)--(4,11);
\foreach \Point in {(0,11), (1,12),(2,12),(3,12),(4,11)}{
 \fill \Point circle[radius=1pt]; }
  \node[align=center] at (8,11) {$(a^2-ab+c)(2b-3a)^2$};
\draw (0,13)--(1,14)--(2,14)--(3,13)--(4,13);
\foreach \Point in {(0,13), (1,14),(2,14),(3,13),(4,13)}{
 \fill \Point circle[radius=1pt]; }
 \node[align=center] at (8,13) {$(a^2-ab+c)(2b-3a)(b-a)$};
\draw (0,15)--(1,16)--(2,15)--(3,15)--(4,15);
\foreach \Point in {(0,15), (1,16),(2,15),(3,15),(4,15)}{
 \fill \Point circle[radius=1pt]; }
\node[align=center] at (8,15) {$(b-a)^2(a^2-ab+c)$};
\draw (0,17)--(1,17)--(2,17)--(3,17)--(4,17);
\foreach \Point in {(0,17), (1,17),(2,17),(3,17),(4,17)}{
 \fill \Point circle[radius=1pt]; }
 \node[align=center] at (8,17) {$(b-a)^4$};
\end{tikzpicture}
\caption{The $9$ Motzkin paths of length $4$ and their contributions to $[x^4]\Rev\left(\frac{1+ax}{1+bx+cx^2}\right)$}\label{Motzkin}
\end{center}
\end{figure}
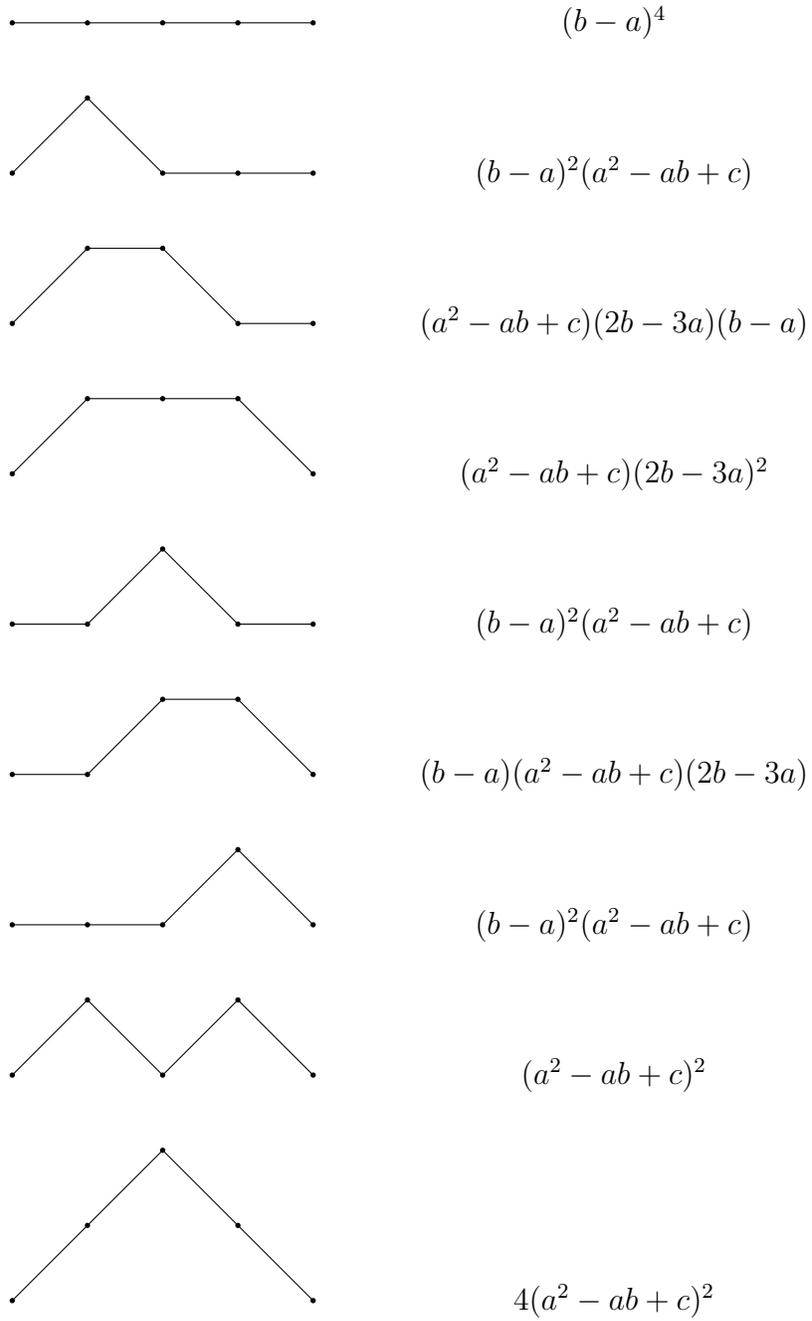

\section{Jacobi and Thron continued fractions}
We have seen that the generating function of the revert transform of $\frac{1+ax}{1+bx+cx^2}$ can be expressed as the Jacobi continued fraction
$$\cfrac{1}{1-(b-a)x-\cfrac{(a^2-ab+c)x^2}{1-(b-2a)x-\cfrac{(a^2-ab+c)x^2}{1-(b-2a)x-\cfrac{(a^2-ab+c)x^2}{1-\cdots}}}}.
$$
It is also possible to express this generating function as a Thron continued fraction. This is the content of the following proposition.
\begin{proposition} The generating function of the revert transform of $\frac{1+ax}{1+bx+cx^2}$ can be expressed as the Thron continued fraction
$$\cfrac{1}{1-qx-\cfrac{sx}{1-rx-\cfrac{sx}{1-rx-\cfrac{sx}{1-rx-\cdots}}}},$$ where
\begin{align*}q&=\frac{b-\sqrt{b^2-4c}}{2}\\
r&=-\sqrt{b^2-4c}\\
s&=-\frac{2a-\sqrt{b^2-4c}-b}{2}\\
\end{align*} and
\begin{align*}
a&=q-r-s\\
b&=2q-r\\
c&=q(q-r).\end{align*}
\end{proposition}
\begin{proof} We solve the equation
$$u=\frac{1}{1-rx-sxu}$$ to obtain
$$u(x)=\frac{1-rx-\sqrt{1-2(r+2s)x+r^2x^2}}{2sx}.$$ Now forming
$$\frac{1}{1-qx-sx u(x)}$$ we obtain the expression
$$\frac{2}{1+(r-2q)x+\sqrt{1-2(r+2s)x+r^2x^2}}.$$
The revert transform of this last expression is
$$\frac{1+(q-r-s)x}{1+(2q-r)x+q(q-r)x^2}.$$
\end{proof}
Note that with $r=-\sqrt{b^2-4c}$, we have $q=\frac{b+r}{2}$ and $s=\frac{b-r}{2}-a$.
\begin{example} \textbf{The Fibonacci numbers}
The generating function of the revert transform of the Fibonacci numbers $F_{n+1}$ \seqnum{A000045} can be expressed as the Thron-type continued fraction
$$\cfrac{1}{1+\frac{\sqrt{5}+1}{2}x-\cfrac{\frac{\sqrt{5}-1}{2}x}{1+\sqrt{5}x-\cfrac{\frac{\sqrt{5}-1}{2}x}{1+\sqrt{5}x-\cfrac{\frac{\sqrt{5}-1}{2}x}{1+\cdots}}}}.$$
\end{example}
\begin{example} The Thron continued fraction that begins
$$\cfrac{1}{1-\frac{\sqrt{5}-1}{2}x+\cfrac{\frac{\sqrt{5}-1}{2}x}{1-\sqrt{5}x+\cfrac{\frac{\sqrt{5}-1}{2}x}{1-\sqrt{5}x+\cfrac{\frac{\sqrt{5}-1}{2}x}{1-\cdots}}}}$$ is the generating function of the revert transform of the sequence that begins
$$1,0,1,1,2,3,5,8,13,\ldots$$ with generating function $\frac{1-x}{1-x-x^2}$.
\end{example}
\begin{example} \textbf{The Jacobsthal numbers}
The Jacobsthal numbers \seqnum{A001045} have generating function $\frac{1}{1-x-2x^2}$ and begin
$$1,1,3,5,11,21,\ldots.$$
Their revert transform is the sequence beginning
$$1, -1, -1, 5, -3, -21, 51, 41, -391,\ldots.$$
The generating function $\frac{\sqrt{1+2x+9x^2}-x-1}{4x^2}$ of this  sequence has the Thron-type continued  fraction expression
$$\cfrac{1}{1-x+\cfrac{2x}{1-3x+\cfrac{2x}{1-3x+\cdots}}}.$$
\end{example}
We note that the revert transform of $[x^n]\frac{1+bx}{1+ax}$ has a generating function expressible as
$$\cfrac{1}{1-ax+\cfrac{bx}{1-ax+\cfrac{bx}{1-ax+\cdots}}}.$$
\begin{example} \textbf{The Catalan numbers $C_n$ and the factorial numbers $n!$} The revert transform of $1-x$ is $c(x)=\frac{1-\sqrt{1-4x}}{2x}$, the generating function of the Catalan numbers $C_n=\frac{1}{n+1} \binom{2n}{n}$ \seqnum{A000108}. In the expression $1-x$ we have $a=-1, b=c=0$. Thus $c(x)$ can be represented by the Thron continued fraction that has $q=r=0, s=1$. This is the well-known continued fraction (a Stieltjes continued fraction) that begins
$$\cfrac{1}{1-\cfrac{x}{1-\cfrac{x}{1-\cdots}}}.$$
The equivalent Jacobi continued fraction begins
$$\cfrac{1}{1-x-\cfrac{x^2}{1-2x-\cfrac{x^2}{1-2x-\cfrac{x^2}{1-2x-\cdots}}}}.$$
Applying $\mathbb{T}^{-1}$ to this we arrive at the generating function
$$\cfrac{1}{1-x-\cfrac{x^2}{1-3x-\cfrac{4x^2}{1-5x-\cfrac{9x^2}{1-7x-\cdots}}}}$$ of $n!$ \seqnum{A000142}.
\end{example}
\begin{example} \textbf{The Schr\"oder numbers} The prototypical example of a Thron-type continued fraction is the continued fraction expression for the generating function $\frac{1-x-\sqrt{1-6x+x^2}}{2x}$ of the Schr\"oder numbers. This generating function is the revert transform of $\frac{1-x}{1+x}$, where $a=-1, b=1$ and $c=0$. The corresponding Thron continued fraction is thus
$$\cfrac{1}{1-x-\cfrac{x}{1-x-\cfrac{x}{1-x-\cdots}}}.$$
The generating function can also be expressed as the Jacobi continued fraction
$$\cfrac{1}{1-2x-\cfrac{2x^2}{1-3x-\cfrac{2x^2}{1-3x-\cdots}}}.$$
The inverse $\mathbb{T}$ transform of this is then given by
$$\cfrac{1}{1-2x-\cfrac{4x^2}{1-5x-\cfrac{8x^2}{1-8x-\cfrac{18x^2}{1-11x-\cdots}}}}.$$
This is the generating function of the sequence that begins
$$1, 2, 6, 26, 150, 1082, 9366, 94586, 1091670, 14174522, \ldots.$$ This is \seqnum{A000629}, which counts the number of ordered set partitions of subsets of $\{1,\ldots,n\}$. Its generating function is
$$\frac{e^x}{2-e^x}.$$
We now note that the reciprocal of this exponential generating function, namely $\frac{2-e^x}{e^x}=e^{-x}(2-e^x)$, is the exponential generating function of the sequence
$$1,-2,2,-2,2,-2,\ldots$$ with ordinary generating function $\frac{1-x}{1+x}$. 
The sequence \seqnum{A000629} is given by the row sums of the exponential Riordan array 
$$\left[\frac{1}{2-e^x}, x\right]=\left[2-e^x, x\right]^{-1}.$$ 
The matrix defined by $\left[2-e^x, x\right]$ begins
$$\left(
\begin{array}{ccccccc}
 1 & 0 & 0 & 0 & 0 & 0 & 0 \\
 -1 & 1 & 0 & 0 & 0 & 0 & 0 \\
 -1 & -2 & 1 & 0 & 0 & 0 & 0 \\
 -1 & -3 & -3 & 1 & 0 & 0 & 0 \\
 -1 & -4 & -6 & -4 & 1 & 0 & 0 \\
 -1 & -5 & -10 & -10 & -5 & 1 & 0 \\
 -1 & -6 & -15 & -20 & -15 & -6 & 1 \\
\end{array}
\right).$$ This is a signed variant of Pascal's triangle. Then the matrix defined by $\left[\frac{1}{2-e^x}, x\right]$ begins
$$\left(
\begin{array}{ccccccc}
 1 & 0 & 0 & 0 & 0 & 0 & 0 \\
 1 & 1 & 0 & 0 & 0 & 0 & 0 \\
 3 & 2 & 1 & 0 & 0 & 0 & 0 \\
 13 & 9 & 3 & 1 & 0 & 0 & 0 \\
 75 & 52 & 18 & 4 & 1 & 0 & 0 \\
 541 & 375 & 130 & 30 & 5 & 1 & 0 \\
 4683 & 3246 & 1125 & 260 & 45 & 6 & 1 \\
\end{array}
\right).$$ 
The first column numbers in this array are the Fubini numbers, \seqnum{A000670}, which count the number of ordered partitions of $[n]$. They have exponential generating function $\frac{1}{2-e^x}$, and their ordinary generating function is given by the Jacobi continued fraction $\mathcal{J}(3n+1; 2(n+1)^2)$. We have 
$$\mathbb{T}(\mathcal{J}(3n+1; 2(n+1)^2))=\mathcal{J}(1,3,3,\ldots; 2,2,2,\ldots)=\frac{1+x-\sqrt{1-6x+x^2}}{4x}.$$ 
Thus the $\mathbb{T}$ transform of the Fubini numbers is given by the little Schr\"oder numbers \seqnum{A001003}. These are the revert transform of the sequence 
$$1,-1,-1,-1,-1,\ldots$$ with ordinary generating function $\frac{1-2x}{1-x}$ and exponential generating function $2-e^x$. 

\end{example}
\section{Eulerian link}
In this section we begin with the generating function
$$f(x)=\frac{1+(q-r-s)x}{1+(2q-r)x+q(q-r)x^2}.$$ We then find that
$$\frac{r e^{qx}}{r+s-se^{rx}}=\frac{1}{\mathcal{L}^{-1}\left(\frac{1}{u} f\left(\frac{1}{u}\right)\right)}.$$
We note that in the case where $r+s=\alpha, s=\beta$, and $q=r=\alpha-\beta$, we obtain the exponential generating function
$$\frac{(\alpha-\beta) e^{(\alpha-\beta)}}{\alpha-\beta e^{(\alpha-\beta)x}},$$ which expands to give
$$1, \alpha, \alpha(\alpha+\beta), \alpha(\alpha^2+4 \alpha \beta + \beta^2), \alpha(\alpha^3+11 \alpha^2 \beta+11 \alpha \beta^2+\beta^3),\ldots.$$
When $\alpha=1$ and $\beta=t$, these are the Eulerian polynomials \cite{Eulerian}.

Returning to the general case, we find that the generating function $\frac{r e^{qx}}{r+s-se^{rx}}$ may be expressed as the Jacobi continued fraction
$$\mathcal{J}(q+s, 3s+q+r, 5s+q+2r, 7s+q+3r,\ldots; s(r+s), 4 s(r+s), 9 s(r+s),\ldots).$$
Thus the $\mathbb{T}$ transform of this can be expressed as the Jacobi continued fraction
$$\mathcal{J}(q+s, r+2s, r+2s, r+2s,\ldots; s(r+s), s(r+s), s(r+s),\ldots).$$
Alternatively, this may be represented by the Thron continued fraction that begins
$$\cfrac{1}{1-qx-\cfrac{sx}{1-rx-\cfrac{sx}{1-rx-\cdots}}}.$$
These last two continued fractions give the generating function of the revert transform of
$$f(x)=\frac{1+(q-r-s)x}{1+(2q-r)x+q(q-r)x^2}.$$

\section{Narayana and Euler}
Our starting point in this section is the generating function
$$g(x)=\frac{1}{1+(y+1)x+yx^2}.$$ Thus we have $a=0, b=y+1$ and $c=y$.
The revert transform of $f(x)$ is given by
$$\Rev(g)(x)=\frac{1-(y+1)x-\sqrt{1-2(y+1)x+(1-y)^2x^2}}{2x^2y},$$ which is the bivariate generating function of the (symmetric) Narayana triangle, that begins
$$\left(
\begin{array}{cccccc}
 1 & 0 & 0 & 0 & 0 & 0 \\
 1 & 1 & 0 & 0 & 0 & 0 \\
 1 & 3 & 1 & 0 & 0 & 0 \\
 1 & 6 & 6 & 1 & 0 & 0 \\
 1 & 10 & 20 & 10 & 1 & 0 \\
 1 & 15 & 50 & 50 & 15 & 1 \\
\end{array}
\right).$$
The generating function $\Rev(g)(x)$ can be represented by the Jacobi continued fraction
$$\Rev(g)(x)=\frac{1}{1-(y+1)x-\cfrac{yx^2}{1-(y+1)x+\cfrac{yx^2}{1-(y+1)x-\cdots}}},$$ and by the Thron continued faction
$$\cfrac{1}{1-x-\cfrac{yx}{1-(1-y)x-\cfrac{yx}{1-(1-y)x-\cfrac{yx}{1+(y+1)x-\cdots}}}}.$$
The generating function $\Rev(g)(x)$ is the $\mathbb{T}$ transform of the continued fraction
$$\cfrac{1}{1-(y+1)x-\cfrac{yx^2}{1-2(y+1)x-\cfrac{4yx^2}{1-3(y+1)x-\cfrac{9yx^2}{1-4(y+1)-\cdots}}}}.$$
This corresponds to the exponential generating function
$$\frac{(y-1)e^{xy}}{y-e^{(y-1)x}},$$ which is the bivariate generating function of the number triangle
that begins
$$\left(
\begin{array}{ccccccc}
 1 & 0 & 0 & 0 & 0 & 0 & 0 \\
 1 & 1 & 0 & 0 & 0 & 0 & 0 \\
 1 & 3 & 1 & 0 & 0 & 0 & 0 \\
 1 & 7 & 7 & 1 & 0 & 0 & 0 \\
 1 & 15 & 33 & 15 & 1 & 0 & 0 \\
 1 & 31 & 131 & 131 & 31 & 1 & 0 \\
 1 & 63 & 473 & 883 & 473 & 63 & 1 \\
\end{array}
\right).$$
The reciprocal of this exponential generating function $\frac{y-exp^{(y-1)x}}{(y-1)e^{xy}}$ corresponds to the ordinary generating function $\frac{1}{1+(y+1)x+yx^2}$.

The number triangle with generating function $\frac{(y-1)e^{xy}}{y-e^{(y-1)x}}$ is \seqnum{A046802}. This triangle can be described using the elements of the  Eulerian triangle $(W_{n,k})$ \cite{Hirz} with generating function $\frac{(y-1)}{y-e^{(y-1)x}}$, which begins
$$\left(
\begin{array}{ccccccc}
 1 & 0 & 0 & 0 & 0 & 0 & 0 \\
 1 & 0 & 0 & 0 & 0 & 0 & 0 \\
 1 & 1 & 0 & 0 & 0 & 0 & 0 \\
 1 & 4 & 1 & 0 & 0 & 0 & 0 \\
 1 & 11 & 11 & 1 & 0 & 0 & 0 \\
 1 & 26 & 66 & 26 & 1 & 0 & 0 \\
 1 & 57 & 302 & 302 & 57 & 1 & 0 \\
\end{array}
\right).$$
The elements $P_n(y)$ of the polynomial sequence
$$1,y+1,1+3y+y^2, 1+6y+6y^2+6y^3,\ldots$$ can then be described by
$$P_n(y)=\sum_{k=0}^n \sum_{j=0}^n \binom{n}{j}A_{j,k}y^{j-k}.$$ This sequence is in fact a moment sequence. This is the content of the next proposition.
\begin{proposition} The sequence $P_n(y)$ is the moment sequence for the family of orthogonal polynomials whose moment matrix is given by the exponential Riordan array
$$\left[\frac{(y-1)e^{xy}}{y-e^{(y-1)x}}, \frac{e^x-e^{xy}}{e^{xy}-ye^x}\right].$$
\end{proposition}
\begin{proof} Using the theory of Riordan arrays, we find that the production matrix of the above exponential Riordan array is tri-diagonal, beginning
$$\left(
\begin{array}{ccccc}
 y+1 & 1 & 0 & 0 & 0 \\
 y & 2 (y+1) & 1 & 0 & 0 \\
 0 & 4 y & 3 (y+1) & 1 & 0 \\
 0 & 0 & 9 y & 4 (y+1) & 1 \\
 0 & 0 & 0 & 16 y & 5 (y+1) \\
\end{array}
\right).$$
\end{proof}
For more about Riordan arrays and their links to the topics of this note, we refer to the next section.

\section{Riordan arrays}
We recall that an ordinary Riordan array \cite{Book, SGWW} is defined by a pair $(g, f) \in \mathcal{F}_0 \times \mathcal{F}_1$. To this pair is associated the matrix $(a_{n,k})_{0 \le n,k \le \infty}$ with general element
$$a_{n,k}=[x^n] g(x)f(x)^k.$$ Here, $[x^n]$ is the functional that extracts the coefficient of $x^n$ when applied to a power series. An exponential Riordan array is defined by a pair $(g, f) \in \mathcal{F}_0^{(e)} \times \mathcal{F}_1^{(e)},$ where
$$\mathcal{F}_0^{(e)}=\{g(x)=g_0+g_1 \frac{x}{1!}+ g_2 \frac{x}{2!} + \cdots\,|\,g_0 \ne 0\},$$
and
$$\mathcal{F}_1^{(e)}=\{f(x)=f_1 \frac{x}{1!}+f_2 \frac{x^2}{2!}+ \cdots\,|\, f_0=0, f_1 \ne 0\}.$$
To distinguish exponential Riordan arrays from ordinary Riordan arrays, we use the notation $[g,f]$ for exponential Riordan arrays. The general element of the matrix representing $[g, f]$ is given by
$$ \frac{n!}{k!} [x^n]g(x)f(x)^k.$$

The \emph{fundamental theorem of Riordan arrays} details how a Riordan array operates on a generating function.
We have $$(g(x), f(x))\cdot h(x)=g(x)h(f(x))$$ with a similar statement for exponential Riordan arrays.

The expansion of the exponential generating function $\frac{r e^{qx}}{r+s-se^{rx}}$ begins
$$1, q + s, q^2 + 2qs + s(r + 2s), q^3 + 3q^2s + 3qs(r + 2s) + s(r^2 + 6rs + 6s^2),\ldots.$$
As a polynomial sequence in $s$, this sequence $P_n(s)$ has a coefficient array determined by the exponential Riordan array
$$\left[e^{qx}, \frac{e^{rx}-1}{r}\right].$$ Note that when $q=0$ and $r=1$, we get the Riordan array $[1, e^{x}-1]$, whose elements are the Stirling numbers of the $2$nd kind.

We find that
\begin{align*}P_n(s)&=\sum_{k=0}^n n![x^n] e^{qx} \left(\frac{e^{rx}-1}{r}\right)^k s^k\\
&= \sum_{k=0}^n k! a_{n,k} s^k,
\end{align*}
where $a_{n,k}=\sum_{i=0}^n \binom{n}{i}q^{n-i}S_2(i,k)r^{i-k}$ is the $(n,k)$-th element of the exponential Riordan array $\left[e^{qx}, \frac{e^{rx}-1}{r}\right]$.

The generating function of the revert transform of $g(x)=\frac{1+(q-r-s)x}{1+(2q-r)x+q(q-r)x^2}$  is given
by $$\frac{2}{1+(r-2q)x+\sqrt{1-2(r+2s)x+r^2x^2}}.$$
This may be expressed as
$$\left(\frac{1}{1-(2q-r)x}, \frac{x(q(q-r)x-q+r+s)}{(1-(2q-r)x)^2}\right)\cdot c(x),$$ where $c(x)=\frac{1-\sqrt{1-4x}}{2x}$ is the generating function of the Catalan numbers. Alternatively, we have that the generating function of the revert transform of $g(x)=\frac{1+(q-r-s)x}{1+(2q-r)x+q(q-r)x^2}$ is given by the first element of the inverse array
$$\left(\frac{1-(q-r-s)x}{1+(r+2s)x+s(r+s)x^2}, \frac{x}{1+(r+2s)x+s(r+s)x^2}\right)^{-1},$$ indicating that the revert transform is the moment sequence for the family of orthogonal polynomials whose coefficient array is given by the Riordan array $$\left(\frac{1-(q-r-s)x}{1+(r+2s)x+s(r+s)x^2}, \frac{x}{1+(r+2s)x+s(r+s)x^2}\right).$$

\begin{proposition} The expansion of the generating function $\frac{re^{qx}}{r+s-se^{rx}}$ is the moment sequence of the family of orthogonal polynomials whose moment matrix is given by the exponential Riordan array
$$\left[\frac{re^{qx}}{r+s-se^{rx}}, \frac{e^{rx}-1}{s+r-se^{rx}}\right].$$
\end{proposition}
\begin{proof}
In effect, the Riordan array of the statement has a production matrix that begins
$$\left(
\begin{array}{ccccc}
 q+s & 1 & 0 & 0 & 0 \\
 s (r+s) & q+r+3 s & 1 & 0 & 0 \\
 0 & 4 s(r+s) & q+2 r+5 s & 1 & 0 \\
 0 & 0 & 9 s(r+s) & q+3 r+7 s & 1 \\
 0 & 0 & 0 & 16 s (r+s) & q+4 r+9 s \\
\end{array}
\right).$$ This shows \cite{Classical} that it is the moment matrix of a family of orthogonal polynomials whose coefficient array is given by $\left[\frac{re^{qx}}{r+s-se^{rx}}, \frac{e^{rx}-1}{s+r-se^{rx}}\right]^{-1}$.
\end{proof}
Note that
$$\left[\frac{re^{qx}}{r+s-se^{rx}}, \frac{e^{rx}-1}{s+r-se^{rx}}\right]^{-1}=\left[\frac{(1+sx)^{\frac{q-r}{r}}}{(1+(r+s)x)^{\frac{q}{r}}}, \frac{1}{r}\ln\left(\frac{1+(r+s)x}{1+sx}\right)\right]$$ is the coefficient matrix of the relevant family of orthogonal polynomials.

\section{List partition transform}
In \seqnum{A133314} \cite{SL1} Copeland defines a ``list partition transform'' of a sequence $a_n$ with exponential generating function $g(x)$ to be the sequence with exponential generating function $\frac{1}{g(x)}$. Thus the sequence $1,-1,0,0,0,\ldots$ is the list partition transform of $n!$. This transform consists thus in associating the first column of the matrix $[g(x),x]$ with the first column of its inverse $\left[\frac{1}{g(x)}, x\right]$. Note that the general term of $[g(x),x]$ is $\binom{n}{k}a_{n-k}$.
\begin{example}
We consider the sequence \seqnum{A000522} of the total number of arrangements of a set of $n$ elements, given by $$a_n=\sum_{k=0}^n \frac{n!}{k!}.$$ This sequence has exponential generating function $\frac{e^x}{1-x}$.
We have $$\left[\frac{e^x}{1-x}, x\right]^{-1}=\left[\frac{1-x}{e^x}, x\right]=\left[e^{-x}(1-x), x\right].$$
In matrix terms, we have
$$\left(
\begin{array}{cccccc}
 1 & 0 & 0 & 0 & 0 & 0 \\
 2 & 1 & 0 & 0 & 0 & 0 \\
 5 & 4 & 1 & 0 & 0 & 0 \\
 16 & 15 & 6 & 1 & 0 & 0 \\
 65 & 64 & 30 & 8 & 1 & 0 \\
 326 & 325 & 160 & 50 & 10 & 1 \\
\end{array}
\right)^{-1}=\left(
\begin{array}{cccccc}
 1 & 0 & 0 & 0 & 0 & 0 \\
 -2 & 1 & 0 & 0 & 0 & 0 \\
 3 & -4 & 1 & 0 & 0 & 0 \\
 -4 & 9 & -6 & 1 & 0 & 0 \\
 5 & -16 & 18 & -8 & 1 & 0 \\
 -6 & 25 & -40 & 30 & -10 & 1 \\
\end{array}
\right).$$
Thus we associate the sequence $a_n$ which begins $1, 2, 5, 16, 65, 326, \ldots$ with the sequence $1,-2,3,-4,-5,\ldots$.
The generating function $\frac{e^x}{1-x}$ of $a_n$ may be expressed as the continued fraction
$$\cfrac{1}{1-2x-\cfrac{x^2}{1-4x-\cfrac{4x^2}{1-6x-\cfrac{9x^2}{1-8x-\cdots}}}}.$$
The $\mathbb{T}$ transform of this is
$$\cfrac{1}{1-2x-\cfrac{x^2}{1-2x-\cfrac{x^2}{1-2x-\cfrac{x^2}{1-2x-\cdots}}}}.$$ This is the generating function of the shifted Catalan numbers $C_{n+1}$, which are the revert transform of the sequence $1,-2,3,-4,\ldots$ with ordinary generating function $\frac{1}{(1+x)^2}$. The generating function of the shifted Catalan numbers $C_{n+1}$ can also be represented as the Thron-type continued fraction
$$\cfrac{1}{1-x-\cfrac{x}{1-\cfrac{x}{1-\cfrac{x}{1-\cdots}}}}.$$
\end{example}

\section{Consecutive patterns}
The discussion in this section is inspired by the work of Elizalde and Noy \cite{Elizalde}. We start with the generating function
$$g(x)=\frac{1-ux}{1-(u-1)x-(u-1)x^2}.$$
We then have
$$\Rev(g)(x)=\frac{1+(u-1)x-\sqrt{1-2(u+1)x+(u-1)(u+3)x^2}}{2x(u-(u-1)x)}.$$ This can be represented by the Jacobi continued fraction
$$\cfrac{1}{1-x-\cfrac{x^2}{1-(u+1)x-\cfrac{x^2}{1-(u+1)x-\cdots}}},$$ and by the Thron-type continued fraction
$$\cfrac{1}{1-qx-\cfrac{sx}{1-rx-\cfrac{sx}{1-rx-\cfrac{sx}{1-rx-\cdots}}}},$$ where
\begin{align*}q&=\frac{1-u-\sqrt{(u-1)(u+3)}}{2},\\
r&=-\sqrt{(u-1)(u+3)},\\
s&=\frac{1+u+\sqrt{(u-1)(u+3)}}{2}.
\end{align*}
\begin{example}
The revert transform of $g(x)=\frac{1-2x}{1-x-x^2}$ has generating function
$$\frac{\sqrt{1-6x+5x^2}-x-1}{2x(x-2)}.$$
This sequence \seqnum{A033321} begins
$$1, 1, 2, 6, 21, 79, 311, 1265, 5275, 22431, 96900,\ldots$$ and it is the binomial transform of Fine's numbers (\seqnum{A000957}). The generating function can be expressed as
$$\cfrac{1}{1-x-\cfrac{x^2}{1-3x-\cfrac{x^2}{1-3x-\cdots}}},$$ or equivalently
$$\cfrac{1}{1+\frac{\sqrt{5}+1}{2}x-\cfrac{\frac{\sqrt{5}+3}{2}x}{1+\sqrt{5}-\cfrac{\frac{\sqrt{5}+3}{2}x}{1+\sqrt{5}x-\cdots}}}.$$
This sequence counts the number of skew Dyck paths of semi-length $n$ which end in a down step. It also counts the number of permutations avoiding the patterns $\{2431,4231,4321\}$, among others.
\end{example}

Reverting to the general case, the generating function $\Rev(g)(x)$ is the bi-variate generating function of the number triangle that begins
$$\left(
\begin{array}{ccccccc}
 1 & 0 & 0 & 0 & 0 & 0 & 0 \\
 1 & 0 & 0 & 0 & 0 & 0 & 0 \\
 2 & 0 & 0 & 0 & 0 & 0 & 0 \\
 4 & 1 & 0 & 0 & 0 & 0 & 0 \\
 9 & 4 & 1 & 0 & 0 & 0 & 0 \\
 21 & 15 & 5 & 1 & 0 & 0 & 0 \\
 51 & 50 & 24 & 6 & 1 & 0 & 0 \\
\end{array}
\right).$$
This is \seqnum{A092107}, which specifies the number of Dyck paths of semi-length $n$ having exactly $k$ $UUU$'s (triple rises) where $U=(1,1)$. The exponential generating function $\frac{1}{g_e(x)}$ is given by
$$\frac{2e^{\frac{1}{2}(1-u+\sqrt{(u-1)(u+3)})x}\sqrt{(u-1)(u+3)}}{1+u+\sqrt{(u-1)(u+3)}+e^{\sqrt{(u-1)(u+3)}x}(-1-u+\sqrt{(u-1)(u+3)})}.$$
This can be expressed as the Jacobi continued fraction
$$\cfrac{1}{1-x-\cfrac{x^2}{1-(2+u)x-\cfrac{4x^2}{1-(3+2u)x-\cfrac{9x^2}{1-(4+3u)-\cdots}}}}.$$
This generating function is the bi-variate generating function of the triangle that begins
$$\left(
\begin{array}{ccccccc}
 1 & 0 & 0 & 0 & 0 & 0 & 0 \\
 1 & 0 & 0 & 0 & 0 & 0 & 0 \\
 2 & 0 & 0 & 0 & 0 & 0 & 0 \\
 5 & 1 & 0 & 0 & 0 & 0 & 0 \\
 17 & 6 & 1 & 0 & 0 & 0 & 0 \\
 70 & 41 & 8 & 1 & 0 & 0 & 0 \\
 349 & 274 & 86 & 10 & 1 & 0 & 0 \\
\end{array}
\right).$$
This is \seqnum{A162975}, which counts the number of permutations of length $n$ with $k$ occurrences of the subword $132$. The effect of the $\mathbb{T}$ transform in this case is seen in going from the consecutive pattern $132$ in permutations to the pattern $UUU$ in Dyck paths.  On the one hand, there is an apparent loss of specificity, but on the other hand (we can go from one Jacobi fraction to the other), this loss is reversible.

This case study shows that heretofore little studied sequences, such as the sequence \seqnum{A152163} with generating function $\frac{1-2x}{1-x-x^2}$ that begins
$$1,-1,0,-1,-1,-2,-3,-5,-8,\ldots$$ may be worth further analysis. We note that
\begin{align*}\frac{1-2x}{1-x-x^2}&=\cfrac{1}{1+x+\cfrac{x^2}{1-2x}}\\
&=\frac{1}{1+\frac{1}{2}x+\cfrac{\frac{1}{2}x}{1-2x}}.\end{align*}
If we consider the generating function $g(x)=\frac{1}{1+x+x^2}$, which expands to give the sequence that begins
$$1, -1, 0, 1, -1, 0, 1, -1, 0, 1, -1,\ldots,$$ we find that $\frac{1}{g_e(x)}$ is the generating function of
the sequence \seqnum{A049774}, which begins
$$1, 1, 2, 5, 17, 70, 349, 2017, 13358, 99377, 822041, \ldots,$$ and which counts the number of permutations not containing the consecutive pattern $123$. The corresponding ordinary generating function is given by the Jacobi continued fraction $\mathcal{J}(1,2,3,\ldots;1,4,9,\ldots)$. The $\mathbb{T}$ transform of this is given by $\mathcal{J}(1,1,1,\ldots;1,1,1,\ldots)$, which is the generating function of the Motzkin numbers. This corresponds to the fact that the Motzkin numbers are the revert transform of $[x^n]\frac{1}{1+x+x^2}$. The Motzkin numbers count the number of Dyck paths of semi-length $n$ with no $UUU$'s.

\section{Other $\mathbb{T}$ transforms}
The $\mathbb{T}$ transform above depends on the sequence $(n+1)^2$ in an obvious way, so is best adapted to Jacobi fractions whose coefficients of $x^2$ are multiples of that sequence. Other Jacobi fractions can depend on other well-known sequences. In addition, we can combine such transformations. The following examples illustrate this.
\begin{example} We let $g(x)=\frac{1}{1+ax+bx^2}$ and we consider the exponential generating function given by
$$G(x;a,b)=\frac{d}{dx} \Rev\left(\int_0^x g(t)\,dt\right).$$
We find that $$G(x;a,b)=\mathcal{J}(a,2a,3a,\ldots; 2b(1,3,6,10,\ldots))=\mathcal{J}(a(n+1);2b \binom{n+2}{2}).$$ The generating function $G(x)$ expands to give a sequence of polynomials in $s$ that begins
$$1, s, 2r + s^2, 8rs + s^3, 16r^2 + 22rs^2 + s^4,\ldots.$$ As a family of polynomials in $s$, this sequence has a coefficient array that begins
$$\left(
\begin{array}{ccccccc}
 1 & 0 & 0 & 0 & 0 & 0 & 0 \\
 0 & 1 & 0 & 0 & 0 & 0 & 0 \\
 2 r & 0 & 1 & 0 & 0 & 0 & 0 \\
 0 & 8 r & 0 & 1 & 0 & 0 & 0 \\
 16 r^2 & 0 & 22 r & 0 & 1 & 0 & 0 \\
 0 & 136 r^2 & 0 & 52 r & 0 & 1 & 0 \\
 272 r^3 & 0 & 720 r^2 & 0 & 114 r & 0 & 1 \\
\end{array}
\right).$$
We have, for instance,
\begin{align*}
G(x;1,1)&=\frac{3}{4} \sec^2\left(\frac{\sqrt{3}x}{2}+\frac{\pi}{6}\right)\quad & \seqnum{A080635}(n+1)\\
G(x;1,2)&=\frac{7}{3\cos(\sqrt{7}x)-\sqrt{7}\sin(\sqrt{7}x)+4} & \seqnum{A234797}(n+1)\\
G(x;2,1)&=\frac{1}{(1-x)^2} & \seqnum{A000142}(n+1)\\
G(x;2,2)&=\frac{1}{1-\sin(2x)} & \seqnum{A000828}(n+1)\end{align*}
The $\mathbb{T}$ transform corresponding to $\binom{n+2}{2}$ then transforms the Jacobi continued fraction for $G(x)$ to the fraction
$$\mathcal{J}(a,a,a,\ldots; 2b,2b,2b,\ldots).$$
This is the generating function of the $a$-th binomial transform of the expansion of
$$\cfrac{1}{1-\cfrac{2bx^2}{1-\cfrac{2bx^2}{1-\cdots}}}.$$
Thus the $\mathbb{T}$ transform, based on $\binom{n+2}{2}$, of $G(x)$ expands to give the sequence with general term
$$\sum_{k=0}^n \binom{n}{k}a^{n-k} (2b)^{k/2}C_{\frac{k}{2}} \frac{1-(-1)^k}{2}.$$
This $\mathbb{T}$ transform based on $\binom{n+2}{2}$ can be regarded as a map from colored increasing trees to colored Motzkin paths.
\end{example}

It is possible to combine different versions of the $\mathbb{T}$ transform. For instance, we can start with the unsigned Genocchi numbers of the first kind of even index \seqnum{A110501}, which begin
$$1, 1, 3, 17, 155, 2073, 38227, 929569, 28820619, 1109652905,\ldots.$$ 
The generating function of these numbers may be expressed as 
$$G(x)=\mathcal{J}((n+1)(2n+1);(n+2)(n+1)^3).$$ 
Then we have 
$$\mathbb{T}_{\binom{n+2}{2}}(\mathbb{T}_{(n+1)^2}(G(x)))=\mathcal{J}(1,4,4,4,\ldots;2,2,2,\ldots)=\Rev\left(\frac{1-3x}{1-2x-x^2}\right).$$ 
This sequence begins 
$$1, 1, 3, 13, 63, 325,\ldots.$$ Its generating function can be expressed as the following Thron-type continued fraction.
$$\cfrac{1}{1+(1+\sqrt{2})x-\cfrac{(2+\sqrt{2})x}{1+2\sqrt{2}x-\cfrac{(2+\sqrt{2})x}{1+2\sqrt{2}x-\cdots}}}.$$ 
We note that the sequence with generating function $\Rev\left(\frac{1-3x}{1-2x-x^2}\right)$ is the binomial transform of the sequence \seqnum{A114710} which counts the number of hill-free Schr\"oder paths of length $2n$ that have no horizontal steps on the $x$-axis. The generating function of \seqnum{A114710} can be represented by the Thron-type continued fraction
$$\cfrac{1}{1+x-\cfrac{x}{1-x-\cfrac{x}{1-x-\cfrac{x}{1-x-\cdots}}}}.$$ 
In similar fashion, we can consider the Genocchi numbers of the second kind \seqnum{A005439} which begin
$$1, 2, 8, 56, 608, 9440, 198272, 5410688, 186043904, 7867739648,\ldots.$$ The generating function of these numbers is given by 
$$G_2(x)=\mathcal{J}(2(n+1)^2;4{\binom{n+2}{2}}^2).$$
We then have 
$$\mathbb{T}_{\binom{n+2}{2}}(\mathbb{T}_{\binom{n+2}{2}}(G_2(x)))=\mathcal{T}(2,4,4,\ldots;1,1,1,\ldots)=\Rev\left(\frac{1-2x}{1-3x^2}\right).$$
The sequence \seqnum{A033543} with generating function $\Rev\left(\frac{1-2x}{1-3x^2}\right)=\frac{2}{1+\sqrt{1-8x+12x^2}}$, counts the number of Motzkin paths of length $n$ in which the $(1,0)$-steps at level $0$ come in $2$ colors and those at a higher level come in $4$ colors. The generating function of this sequence can be represented by the following Thron-type continued fraction.
$$\cfrac{1}{1+\sqrt{3}x-\cfrac{(2+\sqrt{3})x}{1+2\sqrt{3}x-\cfrac{(2+\sqrt{3})x}{1+2\sqrt{3}x-\cdots}}}.$$ 
This sequence is the binomial transform of \seqnum{A033321}, which we have encountered already.  The sequence \seqnum{A033321} is itself the binomial transform of Fine's numbers \seqnum{A000957}. We end this section by noting that the generating function of Fine's numbers can be represented by the Thron-type continued fraction 
$$\cfrac{1}{1+x-\cfrac{x}{1-\cfrac{x}{1-\cfrac{x}{1-\cdots}}}}.$$ 
This form of the generating function illustrates that Fine's numbers are the INVERT(-1) transform of the Catalan numbers (the INVERT$(r)$ transform of $[x^n]g(x)$ is given by $[x^n]\frac{g(x)}{1-rx g(x)}$).

\section{Binomial transform of a special Thron-type continued fraction}
In the above section, we have seen the binomial transform occur a number of times. We thus pose the following question. Let $a_n$ be the expansion of the revert transform of $g(x)=\frac{1+ax}{1+bx+cx^2}$. We now that $a_n$ will have a generating function expressible as 
$$\cfrac{1}{1-qx-\cfrac{sx}{1-rx-\cfrac{sx}{1-rx-\cdots}}}.$$ 
The question we ask is can we express the generating function of the binomial transform of $a_n$, given by $b_n=\sum_{k=0}^n \binom{n}{k}a_k$, as a Thron-type continued fraction. For this, we have the following result.
\begin{proposition} Let $a_n$ be the revert transform of the expansion of $\frac{1+ax}{1+bx+cx^2}$. Then the generating function of the binomial transform $b_n=\sum_{k=0}^n \binom{n}{k}a_k$ is expressible as the following Thron-type continued fraction 
$$\cfrac{1}{1-\tilde{q}x-\cfrac{\tilde{s}x}{1-\tilde{r}x-\cfrac{\tilde{s}x}{1-\tilde{r}-\cdots}}},$$ where 
\begin{align*}\tilde{q}&=\frac{b+1-\sqrt{(b+1)^2-4(c+a)}}{2}\\
\tilde{r}&=-\sqrt{(b+1)^2-4(c+a)}\\
\tilde{s}&=-\frac{2a-b-1-\sqrt{(b+1)^2-4(c+a)}}{2}.\end{align*}
\end{proposition}
\begin{proof} The binomial transform of a revert transform is equal to the revert transform of the INVERT$(-1)$ transform of the original sequence. Now the INVERT$(-1)$ transform of $g(x)=\frac{1+ax}{1+bx+cx^2}$, which is equal to $\frac{g(x)}{1+x g(x)}$, is equal to
$$\frac{1+ax}{1+(b+1)x+(c+a)x^2}.$$ 
The result follows from this.
\end{proof}
\begin{example} The generating function $\frac{\sqrt{1+2x-11x^2}+3x-1}{2x(2-5x)}$ of the revert transform $a_n$ of $\frac{1+2x}{1+3x+5x^2}$ can be expressed as the Thron-like continued fraction

$$\cfrac{1}{1-\left(\frac{3}{2}-\frac{\sqrt{11}i}{2}\right)x-\cfrac{\left(-\frac{1}{2}+\frac{\sqrt{11}i}{2}\right)x}{1+\sqrt{11}ix-\cfrac{\left(-\frac{1}{2}+\frac{\sqrt{11}i}{2}\right)x}{1+\sqrt{11}ix-\cdots}}}.$$
Then the binomial transform of the binomial transform $b_n=\sum_{k=0}^n \binom{n}{k}a_k$ has a generating function expressible as the Thron-type continued fraction
$$\cfrac{1}{1-(2-sqrt{3}i)x-\cfrac{\sqrt{3}ix}{1+2\sqrt{3}ix-\cfrac{\sqrt{3}ix}{1+2\sqrt{3}ix-\cdots}}}.$$
In this case the sequence $a_n$ begins 
$$ 1, 1, 4, 4, 25, 7, 199, -179,\ldots$$ and the sequence $b_n$ begins 
$$1, 2, 7, 20, 70, 218, 763, 2468,\ldots.$$ 
\end{example}
\section{Conclusions} A main result of this note is that the revert transform of $g(x)=\frac{1+ax}{1+bx+cx^2}$ has a generating function that can be expressed both as a Jacobi- and a Thron-type continued fraction. The combinatorial implications of this need further exploration. We have introduced the $\mathbb{T}$ transform as a means of going from the continued fraction representation of the generating function of $\frac{1}{g_e(x)}$ to that of revert transform of $g(x)$. We have seen that under the $\mathbb{T}$ transform, the image of the factorials $n!$ is the Catalan numbers $C_{n+1}$. The factorials $n!$ count all permutations on a set of $n$ elements, while the Catalan numbers count permutations avoiding any three term permutation. Similarly the sequence \seqnum{A049774} counts permutations not containing the consecutive pattern $123$, while its $\mathbb{T}$ transform, the Motzkin numbers, counts Dyck paths that do not contain the pattern $UUU$. This hints at pattern avoidance being an important factor in the application of the $\mathbb{T}$ transform, with patterns losing their specificity under the transform. We further indicate that other similar transforms, based on common sequences, have special combinatorial meaning.

\bigskip
\hrule
\bigskip
\noindent 2010 {\it Mathematics Subject Classification}:
Primary 05A15; Secondary 11B39, 11B83, 1536, 30B10, 30B70.
\noindent \emph{Keywords:} Jacobi continued fraction, Thron continued fraction, $\mathbb{T}$ transform, series reversion, generating function, Riordan array.

\bigskip
\hrule
\bigskip
\noindent (Concerned with sequences
\seqnum{A000045},
\seqnum{A000108},
\seqnum{A000142},
\seqnum{A000522},
\seqnum{A000629},
\seqnum{A000670},
\seqnum{A000828},
\seqnum{A000957},
\seqnum{A001003},
\seqnum{A001045},
\seqnum{A003543},
\seqnum{A005439},
\seqnum{A033321},
\seqnum{A046802},
\seqnum{A049774},
\seqnum{A080635},
\seqnum{A092107},
\seqnum{A110501},
\seqnum{A114710},
\seqnum{A133314},
\seqnum{A152163},
\seqnum{A162975} and
\seqnum{A234797}.

\end{document}